\documentclass[oneside]{amsart}
\usepackage[dvipsnames]{xcolor}
\usepackage{graphicx}
\usepackage{amsmath,amssymb,amsthm}
\usepackage[shortlabels]{enumitem}
\usepackage[hidelinks]{hyperref}
\usepackage{geometry}
\usepackage{bm}
\usepackage[backend=biber, sorting=nyt ]{biblatex}
\usepackage{tikz}
\usetikzlibrary{arrows.meta, bending, positioning, chains, shapes.multipart, calc}

\newtheorem*{conethm*}{The Cone Theorem}
\newtheorem{theorem}{Theorem}[section]
\newtheorem{corollary}[theorem]{Corollary}
\newtheorem{conjecture}[theorem]{Conjecture}
\newtheorem{question}{Question}
\newtheorem{lemma}[theorem]{Lemma}
\newtheorem{proposition}[theorem]{Proposition}
\newtheorem*{thm:locumc}{\autoref{thm:locumc}}
\theoremstyle{definition}
\newtheorem{definition}[theorem]{Definition}
\newtheorem{example}[theorem]{Example}
\theoremstyle{remark}

\newtheorem{case}{Case}

\newcommand{\bb}[1]{\mathbb{\MakeUppercase{#1}}}
\newcommand{\ca}[1]{\mathcal{\MakeUppercase{#1}}}
\newcommand{\tdeg}[1]{\bm{#1}}
\newcommand{\symdif}{\mathbin{\bigtriangleup}}
\newcommand{\nequiv}{\not\equiv}

\newcommand{\cone}{\triangledown}

\begin{document}

\title{Martin's Conjecture in the Enumeration Degrees}
\author[A. Nakid Cordero]{Antonio Nakid Cordero}
\address{Department of Mathematics, University of Wisconsin--Madison \\
  480 Lincoln Drive, Madison, Wisconsin 53706, USA}
\email{\href{mailto:nakidcordero@wisc.edu}{nakidcordero@wisc.edu}}
\urladdr{\href{https://people.math.wisc.edu/~nakidcordero/}{https://people.math.wisc.edu/~nakidcordero}}

\markboth{A. Nakid Cordero}{Martin's Conjecture in the Enumeration Degrees}

\maketitle

\begin{abstract}
  Martin's Conjecture states that every definable function on the Turing degrees is either constant or increasing, and that every increasing function is
  an iterate of the Turing jump. This classification has already been corroborated for the class of uniformly invariant functions and a long-standing
  conjecture by Steel is that every definable function on the Turing degrees is equivalent to a uniformly invariant one. We explore whether a similar
  classification is possible in the enumeration degrees, an extension of the Turing degrees. We show that the spectrum of behaviour is much wider in
  the enumeration degrees, even for uniformly invariant functions. However, our main result is that uniformly invariant functions behave locally as
  nicely as possible: they are constant, increasing, or above the skip operator. As a consequence, we show that there is a definable function in the
  enumeration degrees that is not equivalent to a uniformly invariant one on any cone.
\end{abstract}
        
\section{Introduction}

A core line of research in computability theory has been to understand the role of the jump operator in the structure of the Turing degrees \cite{marksTuringJumpUnique2011}. Along
this line, Martin conjectured---roughly speaking---that the centrality of the jump is a consequence of a lack of alternatives. Indeed, we can interpret
Martin's Conjecture as giving a mathematical definition of what a \emph{natural} degree-invariant construction of a non-computable set is and then claiming
that the only natural constructions are, up to Turing degree, the jump and its iterates.

The obvious obstacle in stating such a conjecture comes from the term ``natural''. Martin isolated two properties common to all known degree-invariant
constructions hoping to characterize them: definability and relativization. The former takes shape as a strong set-theoretic assumption---the Axiom of
Determinacy (AD)---while the latter is a staple of most results in computability theory.

In the Turing degrees, the strength of AD is wielded in a singular form: Turing Determinacy (TD) is the statement ``every set
$\ca{A}\subseteq \ca{P}(\bb{N})$ closed under Turing equivalence either contains a cone of Turing degrees or is disjoint from a cone''. With a slight rephrasing, TD states
that if a property occurs cofinally in the Turing degrees, then there is a cone of degrees with the property; i.e., there is a degree $\tdeg{d}$ such that
every degree above $\tdeg{d}$ has the property. Martin's Cone Theorem \cite{martinTD1968} is that, over $ZF$, $AD$ implies $TD$; whether the converse relation holds is
unknown. Since the Cone Theorem will be a central element of our discussion, we display it for future reference:

\begin{theorem}[Cone Theorem (Martin \cite{martinTD1968})]\label{conethm}
  Assume $AD$. Let $\ca{A}\subseteq \ca{P}(\bb{N})$ be closed under Turing equivalence. Then, either $\ca{A}$ or its complement $\overline{\ca{A}}$ contains a cone of Turing degrees.
\end{theorem}

This theorem gives us a way to compare the asymptotic behaviour of functions from $\ca{P}(\bb{N})$ to $\ca{P}(\bb{N})$ relative to Turing reducibility. Namely, we say
that \emph{$f$ is Turing-below $g$ on a cone}, denoted by $f\leq_T^{\cone}g$, if there is some $A\subseteq\bb{N}$ such that $f(X)\leq_{T} g(X)$ for all
$X\geq_TA$. Similarly, $f$ is \emph{increasing on a cone} if $\operatorname{Id}\leq^{\cone}_T f$ and \emph{constant on a cone} if $f(X)\equiv_Tf(Y)$, for all $X$ and $Y$ in some cone.
    
A common theme in computability theory is that most results \emph{relativize}. This means that a theorem about computable functions will usually remain true
when an oracle is added to the statement. Moreover, the proof of the relativized statement will be the same as the original, except for the inclusion
of the oracle. The simplest example is the theorem that the halting set $K=\{\langle x,e\rangle \mid \phi_e(x)\downarrow\}$ is not computable. We can relativize the definition of the
halting set to any oracle $A$ to obtain a function $A\mapsto K^A$. Then, the proof that $K$ is not computable can be relativized to show that for all
$A$, $K^{A}\nleq_{T} A$. Of course, figuring out what the correct relativized statement is requires some care.

As with the jump, we think of a construction that relativizes as a function that maps a set $A$ to the result of the construction $C^{A}$ relative to
the oracle $A$. To ensure that the function $A\mapsto C^A$ induces a map in the Turing degrees, we only consider functions on the power set of
$\bb{N}$ with the property that Turing equivalent sets are sent to Turing equivalent sets. We call this property \emph{$T$-invariance}. Martin's Conjecture is
that, under $AD$, the only nontrivial $T$-invariant functions are the iterates of the Turing jump. Precisely,

\begin{conjecture}[Martin]
  Assume $AD$. Then
  \begin{enumerate}[label={(MC \arabic*)},labelsep=10pt, labelindent=0.5\parindent,itemindent=0pt,leftmargin=*,font=\upshape]
    \item\label{mc1} If $f\colon\ca{P}(\bb{N})\to \ca{P}(\bb{N})$ is $T$-invariant then $f$ is either increasing on a cone or constant on a cone.
    \item\label{mc2} The increasing $T$-invariant functions in $\ca{P}(\bb{N})$ are well-ordered up to Turing equivalence on a cone. Moreover, the successor in this
    well-order is given by the Turing jump.
  \end{enumerate}
\end{conjecture}

While both parts of the conjecture remain open, plenty of partial results have been obtained for particular classes of functions. Among these, the
class of \emph{uniformly invariant} functions stands out as the only one where the full conjectured has been solved. Here, a $T$-invariant function
$f$ is uniformly invariant if it is accompanied by a \emph{uniformity function} $u$ that, given a pair of indices $\langle j,i\rangle $ of computable functionals that
witness $A \equiv_T B$, outputs a pair of indices $u(\langle j,i\rangle)$ that witness the equivalence $f(A)\equiv_T f(B)$.

\begin{theorem}[Steel \cite{steelClassificationJumpOperators1982}, Slaman \& Steel \cite{slamanDefinableFunctionsDegrees1988}]\label{thm:ssumc}
  Assume $AD$. Both parts of Martin's Conjecture hold for uniformly $T$-invariant functions.
\end{theorem}

This result is significant because we do not know of any $T$-invariant function that is not uniformly $T$-invariant. Kihara and Montalbán \cite{kiharaUniformMartinsConjecture2018} have
argued that uniformity is also a characteristic of natural constructions and, as such, it should be reflected in Martin's Conjecture. Accepting this
argument means that the philosophical question underpinning Martin's Conjecture has already been solved; however, we are left with a different
challenge:

\begin{conjecture}[Steel \cite{steelClassificationJumpOperators1982}]
  Assume $AD$. Every $T$-invariant function is equivalent to a uniformly $T$-invariant function on a cone.
\end{conjecture}

We are interested in Martin's Conjecture in the context of the enumeration degrees due to their close structural relationship to the Turing degrees. A
set $A$ is \emph{enumeration reducible} to $B$, denoted by $A\leq_eB$, if every enumeration of $B$ computes an enumeration of $A$ (see \hyperref[ereducibility]{Definition~\ref{ereducibility}}). The
structure of the enumeration degrees $\mathcal{D}_e$ is then defined in the usual way as the set of equivalence classes of the symmetric closure of enumeration
reducibility. The Turing degrees $\ca{D}_T$ embed into the enumeration degrees \cite{myhillNoteDegreesPartial1961}; moreover, this copy of the Turing degrees is first-order definable in
$\ca{D}_{e}$ (with the language $\{\leq\}$) \cite{DefinabilityTot2016}. So, the $e$-invariant functions provide a broader playground to test the uniqueness of the Turing jump. However, in
the enumeration degrees there is no single candidate to take the place of the Turing jump. In fact, there are two well-studied uniformly $e$-invariant
extensions of the Turing jump to the enumeration degrees: the enumeration jump---originally defined by Cooper \cite{Cooper_1984}---and the enumeration skip, introduced by
Andrews et al.\ \cite{Andrews2019OnDegrees}.  While both operators share common properties with the Turing jump, each of them presents a distinct global behaviour. For
instance, the enumeration jump is increasing while the skip is not, and the skip satisfies Friedberg's jump inversion but the enumeration jump does
not. Given this state of affairs, we explore the extent to which the uniqueness of the Turing jump is inherited by the enumeration jump and the skip,
and whether the phenomenon conjectured by Martin is reflected in the enumeration degrees.

In \hyperref[sec:negative-results]{Section~\ref{sec:negative-results}}, we show that the asymptotic behaviour of $e$-invariant functions is poles apart to that of uniformly $T$-invariant functions. Namely, the
preorder $\leq_e^{\cone}$ of non-constant $e$-invariant functions compared under enumeration reducibility on a cone is far from a prewellorder. For example,
it contains a copy of the enumeration degrees (and hence of the Turing degrees) and it is not locally countable. The ultimate culprit of this
difference is the lack of a Cone Theorem for the enumeration degrees. The main result of the section (\hyperref[thm:continuumcofinal]{Theorem~\ref{thm:continuumcofinal}}) is joint work with
Jacobsen-Grocott; it provides a general way to use the embedding of the Turing degrees to produce a plethora of examples of uniformly $e$-invariant
functions. For instance, we obtain a uniformly $e$-invariant function that is regressive but not constant in every upper cone.

\hyperref[sec:unifinv]{Section~\ref{sec:unifinv}} focuses on uniformly $e$-invariant functions and their uniformity functions. The construction from \hyperref[sec:negative-results]{Section~\ref{sec:negative-results}} produces uniformly
invariant functions only because they are \emph{locally constant}; that is, the image of every enumeration degree is a single set. Note that in the Turing
degrees, TD implies that locally constant functions are constant in some cone. However, through an analysis of uniformity functions, we show that any
countable family of uniformly $e$-invariant functions can be merged under a single uniformity function. As a consequence, the preorder
$\leq_e^{\cone}$ has infinite antichains of uniformly $e$-invariant functions, even if we ignore the locally constant functions.

After all the negative results, we produce a positive result in line with Martin's Conjecture. Adapting Bard's local approach \cite{bardUniformMartinsConjecture2020}, we prove the best
possible local version of part 1 of the Uniform Martin's Conjecture (UMC 1) in the enumeration degrees. Namely, uniformly $e$-invariant functions are
locally either constant, increasing, or above the skip.

\begin{thm:locumc}
  Let $A\subseteq \bb{N}$ and $f\colon\deg_e(A)\to \ca{P}(\bb{N})$ uniformly $e$-invariant. If $f$ is not constant, then
  \[A\leq_e f(A)\quad\mbox{ or }\quad A^\diamond\leq_e f(A).\]
\end{thm:locumc}

As a consequence, we extend part 1 of \hyperref[thm:ssumc]{Theorem~\ref{thm:ssumc}} (UMC 1) to uniformly invariant functions from the Turing degrees to the enumeration degrees. \hyperref[thm:locumc]{Theorem~\ref{thm:locumc}} also provides
insight on the difference between invariance and uniform invariance. In \hyperref[sec:Steelconj]{Section~\ref{sec:Steelconj}}, we produce a Borel
$e$-invariant function that is not equivalent to a Borel uniformly $e$-invariant function on any upper cone by taking advantage of another structural difference between the Turing
degrees and the enumeration degrees: the existence of $\ca{K}$-pairs. The lack of uniformity comes from an incompatibility between the combinatorics of
$\ca{K}$-pairs and \hyperref[thm:locumc]{Theorem~\ref{thm:locumc}}. As far as we know, this is the first such example in any degree structure.

All the theorems assume only $ZF$, unless otherwise stated.

\section{Enumeration Degrees and the Skip}
\label{sec:edegrees}

Enumeration reducibility was first studied as a generalization of Turing reducibility to the case where only partial information about the oracle is
accessible. While in a Turing reduction $A\leq_TB$ we deal with a program that asks finitely many questions about the oracle $B$ to decide if an element
$n$ is in $A$, in an enumeration reduction $A\leq_eB$ we have a program that can only access positive information about $B$ (i.e.\ if $n\in B$, the program
eventually learns this) and is only required to produce positive information about $A$. In other words, we say that $A$ is \emph{enumeration reducible} to
$B$ if there is an algorithm that takes any enumeration of $B$ as input and returns an enumeration of $A$. This algorithm can be thought as a c.e.\
list of axioms of the form \emph{``if $D\subseteq B$, enumerate $n$ into $A$''} for $n \in \bb{N}$ and $D$ finite. This is made precise with the following definition:

\begin{definition}[Friedberg and Rogers, \cite{FriedbergRogers1959}]\label{ereducibility}
  Let $A,B\subseteq \bb{N}$. We say that $A$ is enumeration reducible to $B$, denoted by $A\leq_e B$, if there is a c.e.\ set $\Gamma$ such that
  \[n\in A\quad \mbox{if and only if}\quad \exists u (\langle n, u\rangle \in \Gamma \;\&\; D_u\subseteq B)\,,\] where $\{D_u\}$ is the standard numbering of finite sets.  In this context, we say that
  $\Gamma$ is an enumeration operator (or just an $e$-operator) and we write $A=\Gamma(B)$. To simplify notation, we identify a finite set $D$ with its
  canonical index $u$. So, we write $\langle n, D\rangle \in\Gamma $ instead of $\langle n,u \rangle \in \Gamma $.
\end{definition}

Since the enumeration operators are exactly the c.e.\ sets, we have a standard computable numbering of all enumeration operators
$\{\Gamma_j\}_{j\in\bb{N}}$. When $A=\Gamma_i (B)$, we say that $A\leq_{e}B$ via $i$.

As usual, one can define \emph{enumeration equivalence} by setting $A\equiv_{e }B$ when $A\leq_eB$ and $B\leq_eA$. If $A\leq_e B$ via $i$ and
$B\leq_eA$ via $j$, we say that $A\equiv_e B$ via $\langle i,j\rangle$. Enumeration equivalence is an equivalence relation and induces the degree structure
$\ca{D}_e=\ca{P}(\bb{N})/\equiv_e$ which we call the \emph{enumeration degrees}. It is an upper semi-lattice with least element $\boldsymbol{0}_e$, the degree of c.e.\ sets, and least upper
bound given by the usual join operator.

The enumeration degrees are a natural extension of the Turing degrees, this is witnessed by lifting the map $\iota(A)=A\oplus \overline{A}$ of sets to degrees because
\[A \leq_T B\quad \mbox{if and only if}\quad A \oplus \overline{A}\leq_e B \oplus \overline{B}.\] A set $A$ is called \emph{total} if
$A\geq_e\overline{A}$ (equivalently, if $A\equiv_e\iota(A)$) and an enumeration degree is total if it contains a total set. Thus, the total enumeration degrees form an
isomorphic copy to the Turing degrees inside the enumeration degrees. Moreover, the total enumeration degrees are first-order definable in $\ca{D}_e$ \cite{DefinabilityTot2016}.

However, not every enumeration degree is total \cite{Medvedev}. In fact, the non-total degrees are cofinal. Call an enumeration degree $\tdeg{b}$ a \emph{quasiminimal cover} of a
degree $\tdeg{a}$ if $\tdeg{a}<\tdeg{b}$ and there is no total degree strictly between $\tdeg{a}$ and $\tdeg{b}$. Similarly, $\tdeg{b}$ is a \emph{strong quasiminimal cover} of a degree
$\tdeg{a}$ if $\tdeg{a}<\tdeg{b}$ and every total degree $\tdeg{c}\leq\tdeg{b}$ also satisfies $\tdeg{c}\leq\tdeg{a}$. Sufficiently generic sets produce strong quasiminimal covers. Given a set
$A\in\tdeg{a}$, we say that $G$ is \emph{$\langle A\rangle$-generic} if for every set $W \subseteq 2^{<\omega}$ such that $W\leq_e A$, we have that
\[\exists\sigma \prec G\;\left( \sigma \in W \; \; \lor \; \; \forall \tau\succ\sigma \left( \tau \notin
      W\right)\right).\] This notion extends 1-genericity in the Turing degrees. Namely, $G$ is 1-generic relative to a (total) oracle $A$ if and only
if $G$ is $\langle A\oplus \overline{A}\rangle $-generic. In this case, we simply say that $G$ is $A$-generic.

\begin{proposition}[essentially Medvedev \cite{Medvedev}] \label{prop:generic-quasimin} Let $A,G\subseteq \bb{N}$ such that $G$ is $\langle A\rangle$-generic. Then $G\oplus A$ is a strong quasiminimal cover of $A$.
\end{proposition}

This proposition already produces a failure of the cone theorem in the enumeration degrees, as both the total and the non-total degrees are
cofinal. We will dive into the consequences of this failure in the next section.

There is another class of enumeration degrees that will be of particular interest to us---the \emph{cototal degrees}. They were studied by Andrews et al.\ \cite{Andrews2019OnDegrees}
where they showed a broad range of applications to effective mathematics. A set $A$ is \emph{cototal} if $A\leq_e\overline{A}$ and an enumeration degree is cototal if it
contains a cototal set. Every total degree $\tdeg{a}$ is cototal because if $A\in\tdeg{a}$ is total, then
$A\equiv_eA\oplus\overline{A}\equiv_e\overline{A}\oplus A=\overline{A\oplus\overline{A}}$. However, there are cototal degrees that are not total. Gutteridge \cite{gutteridgeResultsEnumerationReductibility1971} and Sorbi \cite{sorbiQuasiminimalEdegreesTotal1988} independently constructed cototal quasiminimal
degrees (a degree is quasiminimal if it is a quasiminimal cover of $\tdeg{0}_{e}$, the enumeration degree of c.e.\ sets). These constructions can be
relativized to total degrees, showing that the cototal degrees that are not total are also cofinal. Moreover, most sets are not cototal, both in terms
of measure and category. Li \cite{LI_2023} showed that the weakly 3-random sets cannot have cototal degree, concluding that the cototal degrees have measure zero,
and Andrews et al.\ \cite{Andrews2019OnDegrees} proved that the enumeration degrees of 2-generic sets are not cototal, which implies that they are meager.

The Turing jump can be extended to all enumeration degrees \cite{Cooper_1984}. By analogy to the halting set, we define for every $A\subseteq \bb{N}$ the set $K_A$ by
\[K_A=\{\langle i,n\rangle \mid n \in \Gamma _i(A)\}.\] In contrast to the halting set and Turing equivalence, $A\equiv_eK_A$ for any $A$. However, we always have that
$A\ngeq_e \overline{K}_A$. Thus, we define the \emph{enumeration jump} of $A$ as \[A'= K_{A}\oplus \overline{K}_A.\]

\begin{proposition}[Cooper \cite{Cooper_1984}] \label{prop:ejumpProps} Let $A,B\in\ca{P}(\bb{N})$. Then
  \begin{enumerate}[(a)]
    \item \label{prop:jumpincreasing} $A<_e A'$.
    \item\label{prop:jumpmonotone} If $A\leq_e B$, then $A'\leq_e B'$. In particular, the enumeration jump is $e$-invariant.
    \item\label{prop:jumponered} $A\leq_eB$ if and only if $K_A\leq_1 K_B$.
    \item\label{prop:jumptotal} $A'$ is total.
  \end{enumerate}
\end{proposition}

Note that the first three properties match similar results for the Turing jump and the halting set. Moreover, the proof of \ref{prop:jumponered} is uniform, which implies
that the enumeration jump is uniformly $e$-invariant because if $A\equiv_eB$, then $K_A\equiv_1K_B$ and $\overline{K_A}\equiv_1\overline{K_B}$. However, the fact that
$A'$ is always total produces a difference. Since non-total degrees are cofinal, no cone of enumeration degrees is contained in the range of the jump
operator.

An alternative extension of the Turing jump is given by the skip operator. This operator was only recently studied by Andrews et al.\ \cite{Andrews2019OnDegrees}, and while
it also shares many useful properties with the Turing jump, it presents quite a distinct behaviour. The \emph{skip of $A$} is defined as the set
$A^\diamond=\overline{K}_A$ and similarly to the enumeration jump, it is a uniformly $e$-invariant function. The degree of $A^{\diamond }$ is the maximum degree of a set
$\overline{B}$ such that $B\equiv_eA$. From this perspective, it is not surprising that cototality and the skip operator are closely related.

\begin{proposition}[Andrews et al. \cite{Andrews2019OnDegrees}] \label{skipprop} Let $A,B\in\ca{P}(\bb{N})$. Then
  \begin{enumerate}[(a)]
    \item $A^{\diamond }\nleq_e A $
    \item \label{prop:skipincreasing} $A$ has cototal degree if and only if $A\leq_e A^{\diamond }$ if and only if $A^{\diamond }\equiv_eA'$.
    \item \label{prop:skiponered} $A\leq_eB$ if and only if $A^{\diamond }\leq_1 B^{\diamond }$.
    \item \label{prop:skipinv} If $\boldsymbol{0}_e'\leq_e A$, there is a set $G$ such that $G^{\diamond }\equiv_e A$
  \end{enumerate}
\end{proposition}

Even though the skip and the enumeration jump agree on the cototal degrees, they can differ wildly. For example, there are sets in which the skip is
cyclic \cite{Andrews2019OnDegrees}. Namely, there is a set $A$ such that $\left(A^{\diamond}\right)^{\diamond }\equiv_eA$, although any set with this property is fairly complicated (bounds all
hyperarithmetical sets). With this in mind, a naive reading of Martin's Conjecture is false when translated to the enumeration degrees: the skip is a
natural operation on the enumeration degrees that is not constant, not increasing, and not comparable to the jump in every upper cone. A natural
question is then what kind of behaviour can we get from the ``natural'' functions on the enumeration degrees. For example, Kihara and Montalbán \cite{kiharaUniformMartinsConjecture2018} proved
that, under a strengthening of $AD$, the uniformly invariant functions from the Turing degrees to the many-one degrees are not well-ordered when
compared on a cone, but they satisfy the closest thing: their order is isomorphic to the Wadge order on Cantor space. We wish to explore if a similar
behaviour is possible in the enumeration degrees.

\section{Global Failure} \label{sec:negative-results}

There are plenty of natural classes of enumeration degrees that are cofinal and pairwise disjoint. We have already mentioned the total, the non-total
cototal, and the non-cototal degrees; but there many more. Most of these classes can be characterized from topological considerations as every
effective second-countable $T_0$ topological space produces a class of enumeration degrees \cite{Kihara_Pauly_2022}. A particularly interesting example is the class of
continuous degrees, the enumeration degrees of points in Hilbert's cube $[0,1]^{\omega }$, which sits strictly between the total and the cototal
degrees. Moreover, the continuous degrees turn out to be definable in the enumeration degrees \cite{andrewsCharacterizingContinuousDegrees2019}. With this topological approach, it is possible to
isolate countably-many classes of enumeration degrees that are all cofinal and disjoint. For the details, refer to \cite{kiharaEnumerationDegreesNonmetrizable}.

Cofinal classes allow us to stitch together different $e$-invariant functions to define a new one that is incomparable with all of them in any
cone. The following example is the simplest application of this method.

\begin{example} \label{example:counterglobal} Define ${h:\ca{P}(\bb{N})\to \ca{P}(\bb{N})}$ as
  \[h(A)=\begin{cases}A & \mbox{if }A \mbox{ has cototal degree;}\\A^{\diamond} & \mbox{otherwise.}\end{cases}\] Note that both of the maps
  $A\mapsto A$ and $A \mapsto A^{\diamond}$ are $e$-invariant, so $h$ is $e$-invariant. Moreover, for every $X\subseteq \bb{N}$, there are $A,B>_e X$ such that
  $A$ has cototal degree and $B$ does not. In that case, $A\leq_e h(A)$ but $B \nleq_e h(B)$ by \hyperref[skipprop]{Proposition~\ref{skipprop}}.  Similarly,
  $h(A)<_eA^{\diamond}$ but $B^{\diamond}\leq_e h(B)$. So, $h$ is a Borel $e$-invariant function that is not constant, not increasing, and not comparable to either the
  jump or the skip on any upper cone.
\end{example}

This idea can be pushed by finding larger families of pairwise disjoint cofinal sets of enumeration degrees. While the topological approach mentioned
above can be used to combine, for example, countably many $e$-invariant functions to produce infinite antichains in the partial preorder of Borel
$e$-invariant functions compared on a cone, the following result provides a more general and systematic approach. We say that a function
$g\colon \ca{P}(\bb{N})\times\ca{P}(\bb{N})\to\ca{P}(\bb{N})$ is \emph{arithmetical} if for every $X,Y\subseteq \bb{N}$ there is some $n\in\bb{N}$ such that $g(X,Y)\leq_T (X\oplus Y)^{(n)}$.

\begin{theorem}[joint with Jacobsen-Grocott] \label{thm:continuumcofinal} There is an arithmetical function ${g\colon\ca{P}(\bb{N})\times\ca{P}(\bb{N})\to\ca{P}(\bb{N})}$ such that
  \begin{enumerate}[(a)]
    \item $\ca{C}(A)=\{\deg_e \left(g(A,C)\right)\mid C\subseteq \bb{N}\}$ is a cofinal subset of $\ca{D}_e$.
    \item If $A\neq B$, then $\ca{C}(A)\cap \ca{C}(B)=\varnothing$.
  \end{enumerate}
\end{theorem}

\begin{proof}
  For each $A\subseteq \bb{N}$, we will construct a strong quasiminimal cover $M_A$ of $A\oplus \overline{A}$ such that if $A\neq
  B$, then $M_{A}\nequiv_e M_{B}$. After we finish the construction, we can define the function $g(A,C)=M_{A\oplus
    C}$.  Before moving on to the construction, we argue that $g$ meets the conditions of the theorem.

  First, take any $A\subseteq \bb{N}$. We want to verify that $\ca{C}(A)$ is cofinal, so let $C\subseteq \bb{N}$. Then,
  \[C\leq_eA\oplus C\oplus\overline{A\oplus C}<_e M_{A\oplus C}\in \ca{C}(A).\] Now, let $A\neq B$ be sets of natural numbers. If
  ${\ca{C}(A)\cap\ca{C}(B)\neq\varnothing}$, there are $C,D\subseteq \bb{N}$ such that $M_{A\oplus C}\equiv_eM_{B\oplus D}$. By construction, $M_{A\oplus C}\equiv_eM_{B\oplus D}$ only holds if $A\oplus C=B\oplus D$; so $A=B$.

  The main challenge in the construction is to make sure that $M_{A}\nequiv_eM_{B}$ whenever $A\neq B$, while at the same time keeping the construction
  arithmetical. To overcome this, we first give a general argument that guarantees $M_{A}\nequiv_eM_{B}$ for $A$ and $B$ in distinct Turing degrees. Then,
  we will handle the case where $A$ and $B$ are in the same Turing degree in the construction.

  Suppose that $M_A\leq_e M_B$, then $A\oplus \overline{A}<_e M_{A}\leq_e M_{B}$. Since $A\oplus \overline{A}$ is total and $M_{B}$ is a strong quasiminimal cover of
  $B\oplus\overline{B}$, it follows that $A\oplus\overline{A}\leq_eB\oplus\overline{B}$; that is, $A\leq_TB$. Hence, if $A\nequiv_TB$ then
  $M_A\nequiv_e M_{B}$. Now, we proceed to the construction. Since we still need to ensure that $M_{A}\nequiv_e M_{B}$ in the case that $A\neq B$ but
  $A\equiv_T B$, we will simultaneously build $M_A$ for all $A$ in a single Turing degree.

  By \hyperref[prop:generic-quasimin]{Proposition~\ref{prop:generic-quasimin}}, it is enough to build an $A$-generic set $G_A$ to obtain the desired strong quasiminimal cover
  $M_A=A\oplus \overline{A} \oplus G_A$. Let $\{A_k\}_{k\in \bb{N}}$ be a listing of all the elements in $\deg_T(A)$. We construct
  $G=\bigoplus_{k\in \bb{N}} G_{A_k}$ by initial segments $G_s=\bigoplus G_{A_k,s}$. Since the construction requires us to move between Turing reductions and enumeration
  reductions, we fix the following notation for the remaining of the proof: the $n$th Turing operator is denoted by $\Phi _n$ and the $n$th c.e.\ set by
  $W_n$; when using Turing operators and c.e.\ sets, oracles are considered as total. The $n$th enumeration operator (while technically the same as
  the $n$th c.e.\ set) is denoted by $\Gamma _n$.

  We will ensure that the following requirements are satisfied for all $e$, $i$, and $j$ in $\bb{N}$ and every $X\in\deg_T(A)$:
  \begin{align*}
    \mathcal{P}^X_{e,i,j}: & \qquad \left[X=\Phi _i(\Phi_j(X))\; \land\; X\neq \Phi _j(X)\right]\;\;\rightarrow\;\; \left[\Gamma _{e}\left(X\oplus \overline{X} \oplus G_X\right)\neq G_{\Phi _j(X)}\right]\\
    \mathcal{Q}^X_e: & \qquad \exists\sigma \prec G_X\;\left[ \sigma \in W_e^X \;\;\lor\;\; \forall \tau\succ\sigma \left( \tau \notin W_e^X\right)\right]
  \end{align*}

  The $\ca{P}$ requirements ensure that $M_X\nequiv_e M_Y$ whenever $X\neq Y$, and the $\ca{Q}$ requirements ensure that $G_X$ is $X$-generic. The basic strategies to
  satisfy each single requirement are standard. For a $\ca{P}^{X}$ requirement, take the least fresh number $n$ and search for a finite binary sequence
  $H$ extending the current initial segment of $G_X$ such that $n \in \Gamma_e\left(X\oplus \overline{X} \oplus H\right)$. If it exists, extend $G_X$ to
  $H$ and add $n$ to the complement of $G_{\Phi_j(X)}$; otherwise, add $n$ to $G_{\Phi_j(X)}$. For a $\ca{Q}^{X}$ requirement, ask if there is an extension
  $K$ of $G_X$ that meets $W_e^X$ and set $G_X=K$ if it does; otherwise, do nothing.

  Note that to satisfy a $\ca{Q}$ requirement for every $X\in\deg_T(A)$ at the same time, we will add infinitely many elements to $G$, but only finitely many
  at each column. There is, however, a problem that arises from attempting to carry out the construction simultaneously for all $X \in \deg_T(A)$. If
  $X\equiv_T Z$ via $\langle j,i\rangle$, both the requirements $\ca{P}^Z_{e,i,j}$ and $\ca{P}^X_{e,i,j}$ might try to modify $G_X$ simultaneously in incompatible ways. More
  generally, a conflict may arise between the $\ca{P}^X_{e,i,j}$ strategy and the $\ca{P}^Z_{e,i,j}$ strategy when $Z$ is obtained from a repeated application
  of $\Phi_{j}$ to $X$, or when $X$ is obtained from a repeated application of $\Phi_{i}$ to $Z$. We solve this problem by establishing, for each pair
  $\langle i, j\rangle$, a well-founded partial order $\prec_{i,j}$ on $\deg_T(A)$ that ranks the requirements by priority. Namely,
  $\ca{P}^{X}_{e,i,j}$ has higher priority than $\ca{P}^{Z}_{e,i,j}$ if $X\prec_{i,j}Z$. While $\prec_{i,j}$ will not be a linear order, it will resolve the situations
  when a conflict could happen.

  To define $\prec_{i,j}$, we first consider the directed graph $\mathcal{G}_{i,j}$ on $\deg_T(A)$ where an edge goes from $X$ to $Y$ if
  $\Phi _i(X)=Y$ and $\Phi _j(Y)=X$. The in-degree and out-degree of each set $X$ is at most 1, so each connected component is a single vertex, a finite
  path, a finite cycle, or has order type $\omega $, $\omega ^{*}$, or $\omega^{*}+\omega $ (see \hyperref[fig:comp]{Fig.~\ref{fig:comp}}). Sets on different components will be incomparable under
  $\preceq_{i,j}$; this is not a problem because the strategies for the requirements $\ca{P}^{X}_{e,i,j}$ and $\ca{P}^{Z}_{e,i,j}$ do not conflict in this case. We
  define the order on each connected component $\mathfrak{C}$ depending on its type:

  \begin{figure}[h]
    \centering \scalebox{0.8}{
      \begin{tikzpicture}[ >=Latex, node distance=5cm, every node/.style={ circle, draw, thick, inner sep=2pt}, scale=0.5]
        \begin{scope}[yshift=-3cm, xshift=-8cm]
          \node (N1) at ({30}:9cm) {$
            \begin{aligned}
              A\oplus B \oplus C\\
              \phantom{D}\oplus D\oplus E
            \end{aligned}$
          };
          \node (N2) at ({90}:5cm){$
            \begin{aligned}
              E\oplus A \oplus B \\
              \oplus C\oplus D
            \end{aligned}
            $}; \node (N3) at ({150}:8.5cm){$
            \begin{aligned}
              D\oplus E \oplus A \\
              \oplus B\oplus C
            \end{aligned}
            $}; \node (N4) at ({200}:8cm){$
            \begin{aligned}
              C\oplus D \oplus E \\
              \oplus A\oplus B
            \end{aligned}
            $}; \node (N5) at ({340}:8cm){$
            \begin{aligned}
              B \oplus C \oplus D \\
              \oplus E\oplus A
            \end{aligned}$
          };
          \node[draw=none, align=center, inner sep=0pt, shape=rectangle] (C) at (0.3,0.5) {$
            \displaystyle{
              \Phi_i(n)=
              \begin{cases}
                5j+k+1 & \mbox{if } n=5j+k \mbox{ with }k<4 \\
                5j     & \mbox{if } n=5j+4           \\
              \end{cases}
            }$ \\
            \ \\
            $\Phi_j=\Phi_i^{-1}$ }; \draw[->,thick] (N1) to[bend right=20] (N2); \draw[->,thick] (N2) to[bend right=20] (N3); \draw[->,thick] (N3) to[bend
          right=20] (N4); \draw[->,thick] (N4) to[bend right=10] (N5); \draw[->,thick] (N5) to[bend right=20] (N1);
        \end{scope}

      \begin{scope}[local bounding box=case2, start chain, node distance=7mm, xshift=-12cm, yshift=10cm]
        \node[draw=none,inner sep=0pt, shape=rectangle] (C) at (0,-3) { $\displaystyle{
            \begin{aligned}
              \Phi_i(n) & = 2n \\
              \Phi_j(m) &=
                       \begin{cases}
                         0 & \mbox{if }m\mbox{ is odd}\\
                         m/2 & \mbox{if }m \mbox{ is even}
                       \end{cases}
            \end{aligned}
          }$ }; \node [on chain, join={by ->, thick, bend left=20}] at (-5,0){$\mathbb{N}$}; \node [on chain, join={by ->, thick, bend
          left=20}]{$2\mathbb{N}$}; \node [on chain, join={by ->, thick, bend left=20}]{$4\mathbb{N}$}; \node [on chain, join={by ->, thick, bend
          left=20}]{$8\mathbb{N}$}; \node [on chain, join={by ->, thick, bend left=20}, draw=none, fill=none] {$\cdots$};
      \end{scope}

      \begin{scope}[local bounding box=case3, yshift=8cm, start chain=going left, node distance=7mm]
        \node[on chain,draw=none, align=right, inner sep=0pt, shape=rectangle] (E) at (7,0) { $\displaystyle{
            \begin{aligned}
              \Phi_i(n) & =
                       \begin{cases}
                         n+2 & \mbox{if }n \mbox{ is even}\\
                         n-2 & \mbox{if }n>1\mbox{ is odd}\\
                         0 & \mbox{if }n=1
                       \end{cases} \\
              \Phi_j(m)&= \Phi_i^{-1}(m)
            \end{aligned}
          }$ }; \node [on chain=going {at=(\tikzchainprevious),shift=(165:1.5)},draw=none, fill=none, xshift=-0.7cm] {$\cdots$}; \node [on chain, continue
        chain=going below, join={by ->, thick,bend right=20}] {$\{5\}$}; \node [on chain=going {at=(\tikzchainprevious),shift=(320:1)}, join={by
          ->,thick,bend right=20}] {$\{3\}$}; \node [on chain=going {at=(\tikzchainprevious),shift=(340:1)}, join={by ->, thick,bend right=20}]
        {$\{1\}$}; \node [on chain=going {at=(\tikzchainprevious),shift=(280:1)}, join={by ->, thick, bend right=10}] {$\{0\}$}; \node [on chain, join={by
          ->, thick,bend right=20}] {$\{2\}$}; \node [on chain, join={by ->, thick,bend right=20}, draw=none, fill=none] {$\cdots$};
      \end{scope}
    \end{tikzpicture}}
  \caption{Examples of connected components of three different graphs $\ca{G}_{i,j}$ on the degree of computable sets.}
  \label{fig:comp}
\end{figure}
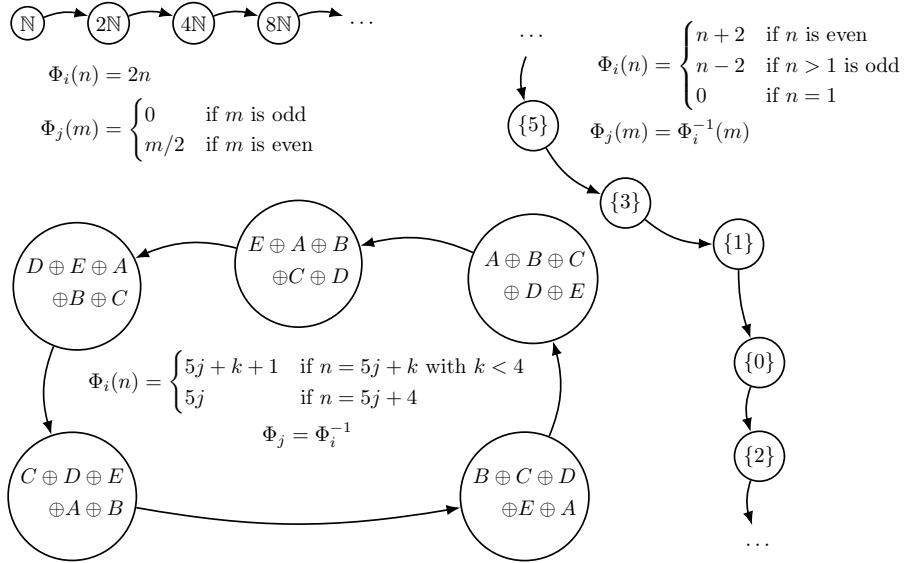
  
  \begin{case}
    If there is some $Y \in \mathfrak{C}$ with in-degree 0 (i.e. if $\mathfrak{C}$ is a finite path or has type $\omega$), then for
    $X,Z\in \mathfrak{C}$ we define $X\prec_{i,j}Z$ if there is a directed path from $X$ to $Z$. In other words, the priority ordering follows the direction of the
    graph.
  \end{case}

  \begin{case}
    If every set in $\mathfrak{C}$ has in-degree 1 but there is $Y \in \mathfrak{C}$ with out-degree 0, then the priority ordering is opposite to the direction of the graph.
  \end{case}

  \begin{case}
    If every set in $\mathfrak{C}$ has in-degree 1 and out-degree 1 ($\mathfrak{C}$ is a cycle or has type $\omega ^{*}+\omega $), let $n$ be the least number with
    $n\in X\setminus Y $ for some $X,Y \in \mathfrak{C}$. Call a set $X$ a \emph{breaking point} of $\mathfrak{C}$ if $n \in X\setminus \Phi_i(X)$. Note that
    $\mathfrak{C}$ necessarily has a breaking point and that breaking points cannot be adjacent in the graph. Since each $\ca{P}^X_{e,i,j}$ requirement only modifies
    $G_X$ and $G_{\Phi_j(X)}$, we can satisfy them for all the breaking points simultaneously without conflict. So, we make the breaking points minimal
    and incomparable in $\prec_{i,j}$. Now, we can simplify the situation to one of the previous case by deleting all the edges that end on a breaking
    point.  Note that every new connected component of $\mathfrak{C}$ belongs to one of the previous cases, so we can define $\prec_{i,j}$ on them accordingly.
  \end{case}

  So, $\prec_{i,j}$ is defined on $\deg_T(A)$, it is a well-founded partial order with height at most $\omega$, and the $\ca{P}$ strategies will not conflict if
  applied sequentially by rank.

  With everything set in place the construction and verification is now straightforward. Start by setting $G_{0}=\varnothing $. At stage
  $s=\langle e, i, j\rangle+1 $, assume that for all $\langle e',i',j' \rangle < \langle e,i,j\rangle$ and all $X\in \deg_T(A)$, the requirements
  $\ca{P}^X_{e',i',j'}$ and $\ca{Q}^X_{e',i',j'}$ have been satisfied. Now, we want to satisfy the requirements $\ca{P}^X_{e,i,j}$ and
  $\ca{Q}^X_{e,i,j}$ for each $X\in\deg_{T}(A)$. We can order $\deg_{T}(A)=\{A_k\mid k\in\bb{N}\}$ by $A_m<A_n$ if $A_{m}$ has strictly lower
  $\prec_{i,j}$-rank than $A_n$, or they have the same $\prec_{i,j}$-rank but $m<n$. Suppose $\ca{P}^X_{e,i,j}$ and $\ca{Q}^X_{e,i,j}$ have been satisfied for all
  $X<A$. Apply the $\ca{P}$-strategy described above to satisfy the $\ca{P}^X_{e,i,j}$ requirement and then the $\ca{Q}$-strategy to satisfy the
  $\ca{Q}^X_{e,i,j}$ requirement.

  A simple inductive argument shows that the construction meets every requirement. So, the function $g(A,C)=M_{A\oplus C}$ is well-defined and satisfies
  the two conditions of the theorem. Now, to show that $g$ is arithmetical, we need to examine the complexity of the construction. For a given $X$, we
  need to be able to compute the set $I=\{i\mid\Phi _i(X)\geq_T X\}$. In general, this can be accomplished consulting the oracle $X^{(3)}$. Then, the listing
  $\{A_k\}_{k\in\bb{N}}$ of $\deg_T(A)$ that we fixed at the beginning of the proof can be taken to be an increasing enumeration of $I$. The same oracle can
  determine whether two sets $Y$ and $Z$ are in the same connected component of $\ca{G}_{i,j}$ (uniformly on $\langle i,j\rangle$). However, we need an extra jump to
  decide the type of each of the connected components and compute $\prec_{i,j}$. Satisfying the requirements only takes one jump, so the whole
  construction of $M_X$ can be carried out with oracle $X^{(4)}$. Thus, $g(A,C)\leq_T (A\oplus C)^{(4)}$.
\end{proof}

The previous theorem gives us a partition of $\ca{D}_e$ into continuum-many cofinal pieces. Namely, the partition whose elements are the sets
$\ca{C}(A)$ and an extra part with all the degrees not in any $\ca{C}(A)$. Since the partition is arithmetical, it allows us to define a Borel
$e$-invariant function as a combination of continuum many functions with different behaviours; this is made precise in the following corollary.

\begin{corollary}\label{cor:jgfunctions}
  For every Borel function $F\colon\ca{P}(\bb{N})\to\ca{P}(\bb{N})$, there is a Borel uniformly $e$-invariant function
  $f\colon\ca{P}(\bb{N})\to\ca{P}(\bb{N})$ such that, for each $A\subseteq\bb{N}$, $\{\deg_e(X)\mid f(X)=F(A)\}$ is cofinal.
\end{corollary}

\begin{proof}
  Define
  \[f(X)=\begin{cases} F(A) & \mbox{if } \deg_e(X)\in \ca{C}(A) \\ \varnothing & \mbox{otherwise.}\end{cases}\] The function $f$ is uniformly
  $e$-invariant for the simple reason that it is locally constant; that is, for all $X$, $f\restriction_{\deg(X)}$ is constant. So, if $\Gamma_e$ is an enumeration
  operator such that $\Gamma _e(Z)=Z$ for all $Z\subseteq \bb{N}$, then the mapping $\langle i,j\rangle\mapsto \langle e,e\rangle $ is a uniformity function for $f$.

  To see that $f$ is Borel, observe that for $X\subseteq \bb{N}$, we can decide if there is some $A$ such that $\deg_e(X)\in \ca{C}(A)$ by first checking if
  $\deg_{e}(X)$ is a strong quasiminimal cover of some total degree $\tdeg{a}$ and, in case it is, verifying if $X\equiv_eM_{Y}$ for some
  $Y\oplus \overline{Y}\in \tdeg{a}$. Checking if $\deg_{e}(X)$ is a strong quasiminimal cover of some total degree $\tdeg{a}$ is $\Sigma ^0_5(X)$: we need to check that
  $X$ does not have total degree---which is $\Sigma^0_3(X)$---and find if there is some total set $Y\oplus\overline{Y}\leq_{e}X$ such that for any enumeration operator
  $\Gamma$ with $\Gamma (X)$ total, we have that $\Gamma(X)\leq_eY\oplus\overline{Y}$. Writing down the statement we get a natural $\Sigma ^0_5$ definition:
  \[\underbrace{\exists e\;\left( \underbrace{\Gamma_e(X)\mbox{ is total}}_{\Pi^0_2(X)} \;\land\;\underbrace{ \forall i\;(\underbrace{\Gamma_i(X)\mbox{ is not total}}_{\Sigma^0_2(X)}\lor
          \underbrace{ \exists j\;(\Gamma_j(\Gamma_e(X))=\Gamma_i(X))}_{\Sigma^0_3(X)})}_{\Pi_0^4(X)}\right)}_{\Sigma^0_5(X)}.\]
  If $X$ is not a strong quasiminimal cover, then $f(X)=\varnothing$. Otherwise, we have to check if $X=M_Y$ for some set $Y\oplus \overline{Y}$ in the degree that
  $X$ is a strong quasiminimal cover of. In the proof of \hyperref[thm:continuumcofinal]{Theorem~\ref{thm:continuumcofinal}}, we saw that $Y^{(4)}$ computes
  $\bigoplus_{Z\in\deg_T(Y)} M_Z$, so $X^{(4)}$ also computes it. Thus, $X^{(5)}$ can answer if there is a $Y$ such that $X=M_Y$ and, if the answer is ``yes'',
  compute $A=\{n\mid 2n\in Y\}$ such that $X\in\ca{C}(A)$; in other words, $X^{(5)}$ can decide in which case of the definition of $f$ we are, so $f$ is Borel.
\end{proof}

\hyperref[skipprop]{Corollary~\ref{cor:jgfunctions}} yields a wider range of behaviours class of $e$-invariant functions than what is possible in the Turing case. For example, Slaman and
Steel \cite{slamanDefinableFunctionsDegrees1988} proved that if a $T$-invariant function is $\emph{regressive}$ (i.e., $f(X)<_TX$) on a cone, then it is constant on a cone. However, this is not the
case for $e$-invariant functions. In fact, the preorder $\leq_e^{\cone}$ is far from a prewellorder.

\begin{corollary}\label{cor:regressive}
  There is a Borel uniformly $e$-invariant function $f \colon \ca{P}(\bb{N})\to\ca{P}(\bb{N})$ such that for all $X,Y\subseteq\bb{N}$ there is
  $Z\subseteq\bb{N}$ such that $X\leq_eZ$ and $f(Z)=Y$ and for every non-c.e.\ set $A$, $f(A)<_{e} A$. In particular, $f$ is not constant on a cone.
\end{corollary}

\begin{proof}
  Take $F$ to be the identity in \hyperref[cor:jgfunctions]{Corollary~\ref{cor:jgfunctions}}. If $X\in\ca{C}(A)$ for some $A$, there is a set $C$ such that
  $X=M_{A\oplus C}$; thus, $f(X)=A<_eM_{A\oplus C}=X$. Otherwise, $f(X)=\varnothing$.
\end{proof}

\begin{corollary}\label{cor:edegemb}
  Enumeration reducibility embeds as a preorder into the set of Borel nonconstant functions preordered by $\leq_e^{\cone}$. In particular, Turing
  reducibility embeds into $\leq_e^{\cone}$.
\end{corollary}

\begin{proof}
  For each $A\subseteq \bb{N}$, define
  \[f_A(X)=\begin{cases} A & \mbox{if } \deg_e(X)\in \ca{C}(\varnothing) \\ \varnothing & \mbox{otherwise.}\end{cases}\] We have that $A\leq_{e}B$ if and only if $f_{A}\leq^{\cone}_e f_B$.
\end{proof}

\begin{corollary}
  The preorder $\leq_e^{\cone}$ is not locally countable, even on uniformly $e$-invariant functions. I.e., there is a uniformly $e$-invariant function
  $f$ such that $\{g\mid g\leq_e^{\cone}f\}$ is uncountable.
\end{corollary}

\begin{proof}
  Let $f$ be the regressive function from \hyperref[cor:regressive]{Corollary~\ref{cor:regressive}}. Note that for each $A\subseteq\bb{N}$, the function
  \[g_A(X)=\begin{cases} A & \mbox{if } \deg_e(X)\in \ca{C}(A) \\ \varnothing & \mbox{otherwise.}\end{cases}\] is uniformly $e$-invariant and $g_A\leq_e^{\cone}f$.
\end{proof}

\section{Uniform Invariance}
\label{sec:unifinv}

The functions obtained from \hyperref[cor:jgfunctions]{Corollary~\ref{cor:jgfunctions}} are uniformly invariant, but only because they are locally constant. To an extent, locally constant
functions play a mirror of the role of constant functions among the uniformly $T$-invariant functions. For example, the constant functions provide an
embedding of $\mathcal{D}_T$ into $\leq_T^{\cone}$ in a similar fashion to our embedding of $\ca{D}_e$ into $\leq_e^{\cone}$ from \hyperref[cor:edegemb]{Corollary~\ref{cor:edegemb}}. Additionally, it is not obvious
whether the function from \hyperref[example:counterglobal]{Example~\ref{example:counterglobal}} is uniformly $e$-invariant or not. Thus, one might hope that a much nicer structure emerges for
$\leq_e^{\cone}$ if we restrict our attention to uniformly $e$-invariant functions that are not locally constant. We obtain a mixed bag of results through
the analysis of uniformity functions. On the one hand, \hyperref[prop:sameunif]{Proposition~\ref{prop:sameunif}} shows that the patching technique from \hyperref[example:counterglobal]{Example~\ref{example:counterglobal}} can be used to merge a
countable collection of uniformly $e$-invariant functions with pairwise disjoint cofinal domains into a single uniformly $e$-invariant function, which
implies that $\leq_e^{\cone}$ has infinite antichains consisting of locally non-constant uniformly $e$-invariant functions. On the other, we will show in \hyperref[sec:locMCe]{Section~\ref{sec:locMCe}}
that any regressive uniformly $e$-invariant function is locally constant, confirming that some of the pathological behaviour of the previous section
is limited to locally constant functions.

Some of the proofs about uniformly $e$-invariant functions involve a good amount of manipulation of indices. In the interest of clarity, we introduce
some convenient notation.

Remember that $A\leq_e B$ via $i$ if $A=\Gamma _i(B)$ and that $A\equiv_e B$ via $\langle i,j\rangle$ if $A\leq_e B$ via $i$ and $B\leq_eA$ via $j$. If
$f$ is an $e$-invariant function, $u\colon\bb{N}\to\bb{N}$ is a uniformity function for $f$ if whenever $A\equiv_e B$ via $\langle i,j\rangle$, we have that
$f(A)\equiv_e f(B)$ via $u(\langle i,j\rangle)$. We write $u(i,j)$ instead of $u(\langle i,j\rangle)$.

The composition of enumeration operators always produces an enumeration operator. Denote by $i\circ j$ an index of the composition of enumeration
operators $\Gamma_i$ and $\Gamma_j$. Namely, $\Gamma_{i\circ j}(A)=\Gamma_i(\Gamma_j(A)$. Similarly, $e^n$ denotes an index for the composition of
$\Gamma_e$ with itself $n$-times. This notation should not be confused for the usual exponentiation of natural numbers, which will never appear in this
context. We extend the notation to pairs of indices $\langle e,i\rangle$ and $\langle j,k\rangle$ by setting
$\langle e,i\rangle\circ \langle j,k\rangle=\langle j\circ e, i\circ k\rangle$ and $\langle e,i \rangle^n=\langle e^n,i^n\rangle$. Observe that the composition of pairs is defined componentwise, but each component is
composed on a different side. This is done to obtain the following lemma.

\begin{lemma} \label{unifcompo} Let $f$ be uniformly $e$-invariant with uniformity function $u$. If $A\equiv_e B$ via $\langle e,i\rangle$ and
  $B\equiv_e C$ via $\langle j,k \rangle$, then $f(A)\equiv_e f(C)$ via $u( e,i )\circ u( j,k )$.
\end{lemma}

\begin{proof}
  Let $\langle e',i'\rangle=u( e,i )$ and $\langle j',k'\rangle=u( j,k )$. By uniformity, $f(A)\equiv_e f(B)$ via $\langle e',i'\rangle $ and
  $f(B)\equiv_e f(C)$ via $\langle j',k'\rangle $. So, $f(C)=\Gamma_{j'}(f(B))=\Gamma_{j'}(\Gamma_{e'}(f(A)))$. Similarly,
  $f(A)=\Gamma_{i'}(f(B))=\Gamma_{i'}(\Gamma_{k'}(f(C)))$.  So, $f(A)\equiv f(C)\mbox{ via } \langle j'\circ e', i'\circ k'\rangle$. In other words, $f(A)\equiv f(C) \mbox{ via } u( e,i )\circ u( j,k )$.
\end{proof}

Having a uniformity function turns out to be a much stronger property than it initially appears. Bard proved, for uniformly $T$-invariant functions,
that the uniformity function can always be chosen to be computable \cite{bardUniformMartinsConjecture2020}. The following lemma states that the same is true in the enumeration degrees; the
proof follows the idea of Bard's, modifying the coding to only use positive information.

\begin{lemma}\label{compunif}
  Let $X\subseteq \ca{P}(\bb{N})$ be closed under enumeration equivalence. If $f\colon X\to \ca{P}(\bb{N})$ is uniformly $e$-invariant, then $f$ has a computable uniformity function.
\end{lemma}

\begin{proof}
  Let $u:\bb{N}\to\bb{N}$ be a uniformity function for $f$ and let $a$, $a^-$, $b$, $b^-$, $c$, and $c^-$ be indices for enumeration operators such that for all
  $A\subseteq \bb{N}$ and all $n,e,i\in \bb{N}$, the following holds:

  \begin{align*}
    \Gamma_a(A)&=A\oplus\{0\}\\
    \Gamma_{a^-}(A\oplus\{0\})&=A\\
    \Gamma_b(A\oplus\{n\})&=A\oplus\{n+1\}\\
    \Gamma_{b^-}(A\oplus\{n+1\})&=A\oplus\{n\}\\
    \Gamma_c(A\oplus \{e\}\oplus\{i\})&=\Gamma_e(A)\oplus \{e\}\oplus\{i\}\\
    \Gamma_{c^-}(A\oplus \{e\}\oplus\{i\})&=\Gamma_i(A)\oplus \{e\}\oplus\{i\}\\
  \end{align*}

  Suppose that $A\equiv_e B$ via $\langle e,i\rangle$, then
  \begin{align*}
    A&\equiv_e A\oplus\{0\}\mbox{ via } \langle a,a^-\rangle \\
    A\oplus\{0\}&\equiv_e A\oplus\{e\}\mbox{ via } \langle b,b^-\rangle ^e\\
    A\oplus\{e\}&\equiv_e (A\oplus\{e\})\oplus\{0\}\mbox{ via } \langle a,a^-\rangle \\
    (A\oplus\{e\})\oplus\{0\}&\equiv_e (A\oplus\{e\})\oplus\{i\}\mbox{ via } \langle b,b^-\rangle ^i\\ 
    (A\oplus\{e\})\oplus\{i\}&\equiv_e (B\oplus\{e\})\oplus\{i\}\mbox{ via } \langle c,c^-\rangle \\
    (B\oplus\{e\})\oplus\{i\}&\equiv_e (B\oplus\{e\})\oplus\{0\}\mbox{ via } \langle b^-,b\rangle ^i\\
    (B\oplus\{e\})\oplus\{0\}&\equiv_e B\oplus\{e\}\mbox{ via } \langle a^-,a\rangle \\
    B\oplus\{e\}&\equiv_e B\oplus\{0\}\mbox{ via } \langle b^-,b\rangle ^e\\
    B\oplus\{0\}&\equiv_e B\mbox{ via } \langle a^-,a\rangle    
  \end{align*}

  So, define $v:\bb{N}\to \bb{N}$ by
  \[v( e,i)=u( a,a^-)\circ u( b,b^-)^e \circ u( a,a^-) \circ u( b,b^-)^i \circ u( c,c^-)\circ u( b^-,b)^e \circ u( a^-,a) \circ u( b^-,b)^e \circ u( a^-,a)\] By \hyperref[unifcompo]{Lemma~\ref{unifcompo}},
  $v$ is a uniformity function for $f$. Moreover, $u( a,a^-)$, $u( a^-,a)$, $u( b,b^-)$, $u( b^-,b)$, $u( c,c^-)$, and $u( c^-,c)$ are six fixed
  integers, so $v$ is computable.
\end{proof}

In particular, countably many uniformity functions are enough to characterize all the uniformly $e$-invariant functions. The next proposition shows
that we can combine a computable sequences of uniformity functions into a single one by just making a slight modification to the corresponding
uniformly $e$-invariant functions.

\begin{proposition}\label{prop:sameunif}
  Let $\{u_n\}_{n\in\bb{N}}$ be a computable sequence of computable uniformity functions. If $\{f_n\}_{n\in\bb{N}}$ is a sequence of uniformly
  $e$-invariant functions such that the uniformity function of $f_i$ is $u_i$, then there is a sequence $\{\hat{f}_n\}_{n\in\bb{N}}$ of uniformly
  $e$-invariant functions and a single uniformity function $u$ for all of them such that for all $A$, $\hat{f}_n(A)\equiv_1 f_n(A)$, uniformly on $n$.
\end{proposition}

\begin{proof}
  Let $u_n( i,j)=\langle i_n,j_n\rangle $ and define $\hat{f}_n(A)=\{\langle n, x\rangle\mid x\in f_n(A)\}$. We have to construct a uniformity function $u$ that works for every $\hat{f}_n$.

  For each $j\in\bb{N}$, let $h_n(j)=h(j,n)$ be an index for the enumeration operator that copies $\Gamma_j$, but only acts on the $n$-th column. Namely,
  \[\Gamma_{h_n(j)}=\{\left\langle \langle n ,m\rangle, \{\langle n, x\rangle\mid x\in D\} \right \rangle\mid \langle m,
    D\rangle\in\Gamma_j\}.\] Note that if $A\equiv_eB$ via $\langle i,j\rangle $, we have that $\hat{f}_n(A)\equiv_e \hat{f}_n(B)$ via
  $\left \langle h_n(i_n), h_n(j_n) \right \rangle $. So, we define the uniformity function $u( i, j )=\langle k,m\rangle $ where $k$ and $m$ are indices for the enumeration
  operators
  \[\Phi =\bigcup_{n\in\bb{N}}\Gamma _{h_n(i_n)}\quad\mbox{ and }\quad\Psi =\bigcup_{n\in\bb{N}}\Gamma
    _{h_n(j_n)}\,,\] respectively. Observe that we need the sequence $\{u_n\}$ to be a computable sequence of computable functions to ensure that
  $\Phi $ and $\Psi $ are indeed enumeration operators.
\end{proof}

The previous lemma lets us recreate \hyperref[unifcompo]{Example~\ref{example:counterglobal}} in the uniform case.

\begin{example}
  Define $h:\ca{P}(\bb{N})\to \ca{P}(\bb{N})$ as \[h(A)=
    \begin{cases}
      A\oplus \varnothing & \mbox{if }A \mbox{ has cototal degree}\\
      \varnothing \oplus A^{\diamond} & \mbox{otherwise.}
    \end{cases}\] Note that both of the maps $A\mapsto A$ and $A \mapsto A^{\diamond}$ are uniformly invariant and, by (the proof of) \hyperref[prop:sameunif]{Proposition~\ref{prop:sameunif}}, $h$ is also
  uniformly invariant. Moreover, $h$ is injective, so it is not locally constant. 
\end{example}

As mentioned before, the previous example can be extended to countably many functions by similar techniques to the previous section.

\begin{corollary}\label{cor:unifantich} There is an infinite antichain
  $\{f_n \mid n\in\bb{N}\}$ in $\leq_e^{\cone}$ consisting of Borel functions that are injective and uniformly $e$-invariant.
\end{corollary}

\section{A Positive Local Result}
\label{sec:locMCe}
We have already seen that the global structure of the $e$-invariant functions is not as tame as Martin conjectured for the $T$-invariant case, even
with the uniformity assumption. However, the failure comes from a topological accident; we can assemble together different functions to create many
unique patchworks on every upper cone. The empirical observation remains: the only known uniformly $e$-invariant functions are combinations of
constants, the identity, and iterates of the skip or the jump. One approach that isolates this problem is to follow Bard's proof \cite{bardUniformMartinsConjecture2020} of part 1 of the
uniform Martin's Conjecture (UMC 1). Adapting his approach to the context of the enumeration degrees, we obtain the best possible local analogue of
(UMC 1) in the enumeration degrees (\hyperref[thm:locumc]{Theorem~\ref{thm:locumc}}). So, despite the negative results in the previous sections, the enumeration jump and the skip might
indeed generate all natural uniformly $e$-invariant functions that are not locally constant, even if the generating process is not just iteration; we
will discuss this further in the last section.

\begin{theorem}\label{thm:locumc}
  Let $A\subseteq \bb{N}$ and $f\colon\deg_e(A)\to \ca{P}(\bb{N})$ be uniformly $e$-invariant. If $f$ is not constant, then
  \[A\leq_e f(A)\quad\mbox{ or }\quad A^\diamond\leq_e f(A).\]
\end{theorem}

While the overall proof follows the argument of Bard, we need a more delicate coding to handle the proof with only positive information. It is
noteworthy that the presence of the skip in the conclusion of the theorem, which is unavoidable, stems from the only part of the argument where
negative information is necessary. We need two lemmas before we proceed to the proof.

\begin{lemma}\label{lem-i}
  Let $X\subseteq \ca{P}(\bb{N})$ be closed under enumeration equivalence and $f\colon X\to\bb{N}$ be uniformly $e$-invariant. For any $A,B\in X$ such that
  $f(A)\neq f(B)$, we have that $f(A\oplus\varnothing)\neq f(B\oplus\varnothing)$ and $f(\varnothing\oplus A)\neq f(\varnothing\oplus B)$.
\end{lemma}

\begin{proof}
  The result is obvious if $A\nequiv_eB$, so suppose that $A\equiv_e B$. We proceed by contrapositive, without loss of generality, suppose that
  $f(A\oplus\varnothing)= f(B\oplus\varnothing)$. Let $i,j\in\bb{N}$ such that
  \begin{align*}
    \Gamma_i&=\{\langle m, \{2m\}\rangle\mid m\in\bb{N} \}\\
    \Gamma_j&=\{\langle 2m, \{m\}\rangle\mid m\in\bb{N} \}.
  \end{align*}
  Then, we have that $A\equiv_e A\oplus\varnothing$ via $\langle i,j\rangle$ and $B\equiv_e B\oplus\varnothing$ via $\langle i,j\rangle$. Since $f$ is uniformly
  $e$-invariant, it has a uniformity function $u$ and there are $k,\ell\in\bb{N}$ such that $u(i,j)=\langle k,\ell \rangle$. Thus,
  \[f(A)=\Gamma_k(f(A\oplus\varnothing))=\Gamma_k(f(B\oplus\varnothing))=f(B).\]
\end{proof}

\begin{lemma}\label{lem-ii}
  Let $A\in \ca{P}(\bb{N})$ and $f \colon \deg_e(A)\to \ca{P}(\bb{N})$ uniformly $e$-invariant and not constant. Then, there is a $B\in\deg_e(A)$ such that
  $f(A\oplus B)\neq f(A\oplus\varnothing)$ or $ f(A\oplus B)\neq f(\varnothing\oplus B)$.
\end{lemma}

\begin{proof}
  Let $C\in\deg_e(A)$ such that $f(C)\neq f(A)$. If $f(A\oplus C)\neq f(A\oplus \varnothing)$ or $f(A\oplus C)\neq f(\varnothing\oplus C)$, we are done by setting
  $B=C$. So, suppose that $f(A\oplus C)=f(A\oplus \varnothing)=f(\varnothing\oplus C)$. By \hyperref[lem-i]{Lemma~\ref{lem-i}}, we have that
  $f(A\oplus\varnothing)\neq f(C\oplus\varnothing)$ and $f(\varnothing\oplus A)\neq f(\varnothing\oplus C)$. Now, we get the result by setting $B=A$ because $f(A\oplus A)\neq f(A\oplus\varnothing)$ or $f(A\oplus A)\neq f(\varnothing\oplus A)$.
\end{proof}

\begin{proof}[Proof of \autoref{thm:locumc}]
  Let $u$ be a computable uniformity function for $f$ and $A\equiv_eB$ via $\langle a,b\rangle $ such that $f(A\oplus B)\neq f(A\oplus\varnothing)$ or
  $f(A\oplus B)\neq f(\varnothing\oplus B)$. Without loss of generality, assume that $f(A\oplus B)\neq f(A\oplus\varnothing)$; the proof is similar in the other case.

  Since $A\equiv_e K_A$, there is $k\in\bb{N}$ such that $\Gamma_{k}(A)=A\oplus K_A$. Define \[X_n=\begin{cases}
    A\oplus B& \mbox{if } n \in K_A\\
    A\oplus \varnothing & \mbox{if } n \notin K_A
  \end{cases}\]

We get that $A\equiv_e X_n$ via $\langle e, h(n)\circ k \rangle $, where
\begin{align*}
  \Gamma_e&=\{\langle m, \{2m\}\rangle\mid m\in\bb{N} \}\\
  \Gamma_{h(n)}&=\{\langle 2m,\{2m\}\rangle\mid m\in\bb{N}\}\cup\{\langle 2m+1,D\oplus \{n\}\rangle\mid \langle m, D\rangle\in \Gamma_b\}
\end{align*}

Since both $u$ and $h$ are computable, we can enumerate $f(X_n)$ from $f(A)$ uniformly on $n$. To see this, let $g_0(n),g_1(n)\in\bb{N}$ such that
$u(e, h(n)\circ k )=\langle g_0(n), g_1(n)\rangle$. Then $f(X_n)\leq_ef(A)$ via $\Gamma_{g_1(n)}$. Hence, $f(A)\geq_e \bigoplus_{n\in\bb{N}}f(X_n)$.

Now, let $M\in\bb{N}$ be the least element of $f(A\oplus B)\symdif f(A\oplus\varnothing)$ and
\[ \Psi = \left\{\left\langle n, \left\{\left\langle n, M \right\rangle\right\}\right\rangle\mid n\in\bb{N}\right\}\]

The enumeration operator $\Psi$ decodes $K_A$ or $\overline{K_A}$ from $\bigoplus f(X_n)$ depending on which side of the symmetric difference $M$ belongs to, that is,
$\Psi\left(\bigoplus f(X_n)\right)=K_A$ if $M\in f(A\oplus B)\setminus f(A\oplus \varnothing)$ and $\Psi\left(\bigoplus f(X_n)\right)=\overline {K_A}$ if $M\in f(A\oplus \varnothing)\setminus f(A\oplus B)$.

Hence, we have that
\[A\equiv_e K_A\leq_e\bigoplus_n f(X_n)\leq_e f(A)\quad\mbox{or}\quad A^\diamond=\overline {K_A}\leq_e\bigoplus_n f(X_n)\leq_e f(A).\]

\end{proof}

We have two easy but informative consequences of \hyperref[thm:locumc]{Theorem~\ref{thm:locumc}}. The first is that the regressive function from \hyperref[cor:regressive]{Corollary~\ref{cor:regressive}} has to be locally constant
if we want it to be uniformly $e$-invariant. The second is a slight extension of (UMC 1) to uniformly invariant functions from the Turing degrees to
the enumeration degrees.

\begin{corollary}\label{cor:unifregresiveconstant} If $f$ is uniformly $e$-invariant and $f(A)<_{e}A$, then $f\restriction_{\deg_e(A)}$ is constant. Thus, the only regressive uniformly $e$-invariant
  functions are locally constant.
\end{corollary}

Call a function $f\colon\ca{P}(\bb{N})\to \ca{P}(\bb{N})$ $(T,e)$-invariant if for all $A,B\in\ca{P}(\bb{N})$, if $A\equiv_T B$, then
$f(A)\equiv_e f(B)$. We can define uniformly $(T,e)$-invariant functions in the obvious way. The advantage of $(T,e)$-invariance is that we can use the
Cone Theorem on the domain to get a global result.

\begin{corollary}
  Assume Turing Determinacy (TD). If $f\colon\ca{P}(\bb{N})\to \ca{P}(\bb{N})$ is uniformly $(T,e)$-invariant, there is $X\subseteq \bb{N}$ such that
  $A\leq_e f(A)$ for all $A\geq_T X$ or $f(A)=f(B)$ for all $A,B\geq_T X$.
\end{corollary}

\begin{proof}
  By the Cone Theorem, the set $\ca{X}=\{ A\subseteq \bb{N} \mid f \restriction_{\deg_T(A)} \mbox{ is constant }\}$ either contains a cone or is disjoint from a cone. If
  $\ca{X}$ contains a cone, then for each $n$ either the set $\ca{Y}_{n}=\{A\in\ca{X}\mid n\in f(A)\}$ contains a cone or $\ca{X}\setminus\ca{Y}_n$ contains a cone. So, there is some
  $X\in\ca{X}$ such that $f(Y)=f(X)$ for all $Y\geq_TX$. Otherwise, if $\ca{X}$ is disjoint from a cone, then there is some set $Z\subseteq \bb{N}$ such that for all
  $Y\geq_{T}Z$, $f \restriction_{\deg_T(Y)}$ is not constant.

  Now, let $\mathcal{T}=\{ A\oplus \overline{A}\mid A\subseteq\bb{N}\}$. The function $f$ induces a uniformly $e$-invariant function
  $\hat{f}:\mathcal{T}\to \ca{P}(\bb{N})$ given by $\hat{f}(A\oplus \overline{A})=f(A)$. Moreover, we have that
  $\hat{f}\restriction_{\deg_e(A\oplus \overline{A})}$ is not constant for any $A\oplus \overline{A}\geq_eZ\oplus \overline{Z}$. It is straightforward to modify the proof of \autoref{thm:locumc} to accept functions only defined on sets
  of the form $A\oplus \overline{A}$. Thus, $\hat{f}(A\oplus \overline{A})\geq_e A\oplus \overline{A}$ or $\hat{f}(A\oplus \overline{A})\geq_e (A\oplus \overline{A})^{\diamond }\equiv_e (A\oplus \overline{A})'$. In any case, $A\leq_eA\oplus \overline{A}\leq_e \hat{f}(A\oplus \overline{A})=f(A)$.
\end{proof}

\section{A Non-Uniformly Invariant Function}
\label{sec:Steelconj}
A point of interest around Martin's Conjecture is the dividing line between uniformly and non-uniformly $T$-invariant functions (on a cone). Early on,
Steel conjectured that, under $AD$, they are exactly the same \cite{steelClassificationJumpOperators1982}. Steel's conjecture remains open and a positive result implies Martin's Conjecture. An
unexpected consequence of \hyperref[thm:locumc]{Theorem~\ref{thm:locumc}} is that Steel's conjecture fails for $e$-invariant functions. Notably, we will construct a Borel example using the notion of \emph{Kalimullin pair}, a technical tool generalizing semicomputable sets that has been repeatedly used to obtain definability
results in the enumeration degrees \cite{kalimullin2003, DefinabilityTot2016, GANCHEV_KALIMULLIN_MILLER_SOSKOVA_2022}.

\subsection{Semicomputable sets and $\ca{K}$-pairs}

Semicomputable sets were introduced by Jockusch \cite {jockuschReducibilitiesRecursiveFunction1966} to study the interaction between strong reducibilities. A set
$A$ is \emph{semicomputable} if $A$ is a left cut of a computable linear order. Note that if $A$ is semicomputable, $\overline{A}$ also is\footnote{If
  $A$ is a left cut of a computable linear order, then $\overline{A}$ is a left cut of the inverse order.}.  We say that $\{A,\overline{A}\}$ form a \emph{semicomputable pair}. Jockusch \cite{jockuschReducibilitiesRecursiveFunction1966} proved that every Turing degree contains a nontrivial semicomputable pair; that is, one
where $A$ and $\overline{A}$ are not c.e. In the enumeration degrees, the halves of nontrivial semicomputable pairs have different degrees and every total degree
is the join of the degrees of the two halves of a semicomputable pair. Moreover, if $A$ is semicomputable then its enumeration degree is
\emph{quasiminimal}---the only total degree it bounds is $\boldsymbol{0}_e$---and forms a strong type of minimal pair with $\overline{A}$: the degree of any set
$X$ is the greatest lower bound of the degrees of $A\oplus X$ and $\overline{A} \oplus X$ \cite{arslanovSplittingPropertiesTotal2003a}. Kalimullin then introduced the
$\ca{K}_U$-pairs (under the name of $U$-$e$-ideal pairs, for reasons that will become apparent below) to prove that the enumeration jump is definable
\cite{kalimullin2003}. They generalize the properties of semicomputable pairs in the context of enumeration degrees.

\begin{definition}[Kalimullin \cite{kalimullin2003}] \label{def:kpair} The pair $\{A,B\}$ forms a \emph{$\ca{K}_U$-pair} if there is $W\leq_e U$ such that
  \[A \times B \subseteq W \qquad \& \qquad \overline{A} \times \overline{B} \subseteq \overline{W}.\] We say that the $\ca{K}_U$-pair is nontrivial if
  $A \nleq_e U$ and $B \nleq_e U$. If $U=\varnothing$ we simply say that $\{A,B\}$ is a $\ca{K}-pair$.
\end{definition}

Being half of $\ca{K}_U$-pair turns out to be a degree-theoretic notion, and the remarkably simple characterization obtained by Kalimullin hints at their
usefulness in definability results.

\begin{proposition}[Kalimullin~\cite{kalimullin2003}]\label{kpair:props} Let $A,B,C,U\in\ca{P}(\bb{N})$
  \begin{enumerate}
    \item\label{prop:kpairdown} If $\{ A, B \}$ is a $\ca{K}_U$-pair and $C\leq_e B$, then $\{ A, C \}$ is a $\ca{K}_U$-pair.
    \item\label{prop:kpairjoin} If $\{ A, B \}$ is a $\ca{K}_U$-pair and $\{ A, C \}$ is a $\ca{K}_{U}$-pair, then $\{ A , B \oplus C\}$ is a $\ca{K}_U$-pair.
    \item $\{ A, B \}$ is a $\ca{K}_U$-pair if and only if \[\forall X\; \deg_e (X)=\deg_e(A \oplus U \oplus X )\land \deg_e(B\oplus U\oplus X).\]
  \end{enumerate}
\end{proposition}

Observe that $\ca{K}_U$-pairs are symmetric, so the proposition above holds equally if we interchange the roles of $A$ and $B$. The first item implies that
the notion of $\ca{K}_U$-pairs is degree-theoretic, and the first and second combined state that set of \emph{$\ca{K}_U$-partners} of half of a
$\ca{K}_U$-pair forms an ideal in the enumeration degrees. The converse of the third item also holds, which gives the definability of $\ca{K}$-pairs in the
enumeration degrees. Moreover, it shows that a $K_U$-pair $\{A,B\}$ is also a $K_V$-pair for any $V\geq_{e}U$ and that it is nontrivial if we also have
that $V\ngeq_eA,B$. If we restrict our attention to nontrivial $\ca{K}_U$-pairs, we can extract more information.

\begin{proposition}[Kalimullin~\cite{kalimullin2003}]\label{nontrivkpair:props}
  Let $\{ A, B \}$ be a nontrivial $\ca{K}_U$-pair. Then
  \begin{enumerate}[(a)]
    \item\label{prop:kpairskip} $A \leq_{e} \overline{B} \oplus U$ and $\overline{A} \leq_e B \oplus U^{\diamond}$. In particular, $A \leq_{e} B^{\diamond} \oplus U$ and $A^{\diamond} \leq_e B \oplus U^{\diamond}$.
    \item\label{prop:kpairquasim} The set $A \oplus U$ is a strong quasiminimal cover of $U$.
  \end{enumerate}
\end{proposition}

A nontrivial semicomputable pair $\{A,\overline{A}\}$ is not only a $\ca{K}$-pair, but a \emph{maximal $\ca{K}$-pair} \cite{ganchevDefinabilityKalimullinPairs2015}: if
$C\geq_eA$, $D\geq_e \overline{A}$, and $\{C, D\}$ is a $\ca{K}$-pair, then $A\equiv_e C$ and $\overline{A}\equiv_eD$. Cai et al.\ \cite{DefinabilityTot2016} proved that these are the only maximal
$\ca{K}$-pairs up to enumeration degree (we define maximal $\ca{K}$-pairs to always be nontrivial), concluding that total degrees are definable. This can be
relativized with a total oracle $T$ to obtain that the degrees of maximal $\ca{K}_T$-pairs are exactly the degrees of nontrivial semicomputable pairs
relative to $U$. In general, if $\{A,\overline{A}\}$ is a semicomputable pair and $A,\overline{A}\nleq_eU$, we have that
$\{A\oplus U,\overline{A}\oplus U\}$ is a maximal $\ca{K}_U-pair$: if $\{C,D\}$ is a (necessarily nontrivial) $\ca{K}_U-pair$ with $A\oplus U\leq_e C$ and
$\overline{A}\oplus U\leq_e D$, by \hyperref[prop:kpairdown]{Proposition~\ref{kpair:props}.\ref{prop:kpairdown}}, $\{A,D\}$ is a nontrivial
$\ca{K}_U$-pair and by \hyperref[prop:kpairskip]{Proposition~\ref{nontrivkpair:props}.\ref{prop:kpairskip}} we have that
$D\leq_e \overline{A}\oplus U$. Similarly, $C\leq_eA\oplus U$. It is not known if all maximal $\ca{K}_U$-pairs can be obtained in this way when $U$ is not total.

\subsection{A Borel Invariant Function that is not Uniformly Invariant}

Now we have all the ingredients to construct the $e$-invariant function that is not uniformly invariant. The following lemma provides the
combinatorial core of the counterexample: a uniformly $e$-invariant function that maps one half of a $\ca{K}_U$-pair to one of its $\ca{K}$-partners must be
constant on that degree.

\begin{lemma} \label{lemma:kpair-non-uniformity} If $A$, $B$, and $U$ are such that $\{A,B\}$ is a nontrivial $\ca{K}_U$-pair and $A>_eU$, every uniformly
  $e$-invariant function $f$ that satisfies $f(A)\equiv_eB$ is constant on $\deg_e(A)$.
\end{lemma}

\begin{proof}
  Let $\tdeg{a}=\deg_e(A)$, $\tdeg{b}=\deg_e(B)$, and $\tdeg{u}=\deg_e(U)$. Assume towards a contradiction that $f$ is uniformly $e$-invariant and not constant on
  $\tdeg{a}$. By \hyperref[thm:locumc]{Theorem~\ref{thm:locumc}}, we have that $\boldsymbol{a} \leq \boldsymbol{b}$ or $\boldsymbol{a}^{\diamond}\leq \boldsymbol{b}$.

  If $\boldsymbol{a} \leq \boldsymbol{b}$, then
  $\boldsymbol{u}= (\boldsymbol{a}\vee\boldsymbol{u}) \wedge (\boldsymbol{b} \vee \boldsymbol{u}) = \boldsymbol{a} \vee \boldsymbol{u}$. This means that
  $\boldsymbol{a} \leq \boldsymbol{u}$, so $\{ A, B\}$ is a trivial $\ca{K}_U$-pair, contrary to our assumption. If instead
  $\boldsymbol{a}^{\diamond}\leq \boldsymbol{b}$, the monotonicity of the skip implies that $\tdeg{u}^{\diamond}\leq \tdeg{a}^{\diamond}$. Combined with \hyperref[prop:kpairskip]{Proposition~\ref{kpair:props}.\ref{prop:kpairskip}},
  \[\boldsymbol{u}<\boldsymbol{u}'= \boldsymbol{u}^{\diamond}\vee \boldsymbol{u} \leq \boldsymbol{a}^{\diamond }\vee \boldsymbol{u} \leq \boldsymbol{b} \vee
    \boldsymbol{u}.\] Since $\boldsymbol{u}'$ is total, we have contradicted the fact that $\boldsymbol{b} \vee \boldsymbol{u}$ is a quasiminimal cover of $\boldsymbol{u}$.
\end{proof}

So, to prove that a function is not uniformly $e$-invariant, it is enough to build a function that maps a cofinal set of halves of
nontrivial relativized $\ca{K}$-pair to their $\ca{K}$-partners, we devote the rest of the section to the construction of a counterexample that works on every
cone and to the proof that no Borel function is equivalent to it on a cone. The main obstacle for the construction is finding, on input $A$, if there are sets $B$ and $U$ such that $\{A, B\}$ is a nontrivial $\ca{K}_U$-pair so that we can map
$A$ to $B$ while at the same time ensuring that if $C \equiv_e A$, we make a consistent choice $D\equiv_e B$ as the output on $C$. The second part can be
managed by restricting our attention to maximal $\ca{K}_U$-partners, as those inhabit a single degree. The next lemma takes care of the first part; the key
idea is to ensure that a unique $U$ will work by only looking for $\ca{K}_U$-pairs $\{A,B\}$ where $U$ has total degree and $U<_eA,B$.

\begin{lemma}
  Let $\boldsymbol{t}_0$ and $\boldsymbol{t}_1$ be total degrees and $\boldsymbol{a} > \boldsymbol{t}_0, \boldsymbol{t}_1$. If
  $\{\boldsymbol{a}, \boldsymbol{b}\}$ is a nontrivial $\mathcal{K}_{\boldsymbol{t}_0}$-pair, and $\{\boldsymbol{a},\boldsymbol{c}\}$ is a nontrivial
  $\mathcal{K}_{\boldsymbol{t}_1}$-pair, then $\boldsymbol{t}_0= \boldsymbol{t}_1$. In particular, $\{ \boldsymbol{a} , \boldsymbol{b} \vee \boldsymbol{c} \}$ is a nontrivial $\mathcal{K}_{\boldsymbol{t}_0}$-pair.
\end{lemma}
\begin{proof}
  $\boldsymbol{a}$ is a quasiminimal cover of both $\boldsymbol{t}_0$ and $\boldsymbol{t}_1$, but $\boldsymbol{t}_0 \vee \boldsymbol{t}_1\leq \boldsymbol{a}$. Since $\boldsymbol{t}_0 \vee \boldsymbol{t}_1$ is total, we have that $\boldsymbol{t}_0= \boldsymbol{t}_0 \vee \boldsymbol{t}_1= \boldsymbol{t}_1$.
\end{proof}

In the previous lemma, it is important that $\boldsymbol{a} > \boldsymbol{t_0}, \boldsymbol{t_1}$. In general, if $\{\boldsymbol{a}, \boldsymbol{b}\}$ is a nontrivial
$\ca{K}_{\boldsymbol{u}}$-pair and $\boldsymbol{u}\vee\boldsymbol{t}\ngeq \boldsymbol{a},\boldsymbol{b}$, then $\{ \boldsymbol{a} , \boldsymbol{b}\}$ is a nontrivial
$\ca{K}_{\boldsymbol{u}\vee\boldsymbol{t}}$-pair. However, if $\boldsymbol{u}$ is total and $\tdeg{a}>\tdeg{u}$, then
$\{\boldsymbol{a}, \boldsymbol{b}\}$ is not a $\ca{K}_{\boldsymbol{v}}$-pair for any $\boldsymbol{v}<\boldsymbol{u}$ because $\boldsymbol{a}$ is not a quasiminimal cover of
$\boldsymbol{v}$. That said, the condition $\tdeg{a}>\tdeg{u}$ is not very restrictive. In general, if $\{\boldsymbol{a}, \boldsymbol{b}\}$ is a nontrivial
$\ca{K}_{\boldsymbol{u}}$-pair, then $\{\boldsymbol{a} \vee \boldsymbol{u}, \boldsymbol{b} \vee \boldsymbol{u}\}$ is a nontrivial $\ca{K}_{\boldsymbol{u}}$-pair. This implies that if $\{\boldsymbol{a}, \boldsymbol{b}\}$ is a maximal $\ca{K}_{\boldsymbol{u}}$-pair, then $\tdeg{a}>\tdeg{u}$ and $\tdeg{b}>\tdeg{u}$. 

\begin{theorem}\label{invnotunif}
  There is a Borel $e$-invariant function $F\colon\ca{P}(\bb{N})\to \ca{P}(\bb{N})$ such that $F$ is not equivalent to any Borel uniformly $e$-invariant function on any cone.
\end{theorem}

\begin{proof}
  We define a function $F$ such that
  \[
    F(A)=\begin{cases}B &\mbox{if }\{A,B\}\mbox{ is maximal }\ca{K}_{U}\mbox{-pair, with }U\mbox{ total.}\\ A &\mbox{otherwise.}\end{cases}
  \]
  Of course, we need to ensure that it is single-valued and Borel. We will construct the function $F$, on input $A$, according to the following
  instructions. First, search for a total set $X\oplus\overline{X}$ below $A$ and a candidate $B$ below $A^{\diamond}$ such that $\{A,B\}$ is a maximal
  $K_{X\oplus\overline{X}}$-pair. If found, map $A$ to $B$. Otherwise, map $A$ to $A$.

  While the procedure to obtain $F(A)$ from $A$ is not computable, it is arithmetical. To make things precise, we define $F(A)$ by the following
  procedure:
  \begin{enumerate}
    \item Search for the least $\langle e, i \rangle$ such that $2n+1\in \Gamma_i(A)$ if and only if $2n\notin \Gamma_i(A)$ and $\{ A,\Gamma_e(A^\diamond)\}$ is a maximal $\mathcal{K}_{\Gamma_i (A)}$-pair.
    \item If such $\langle e, i \rangle$ exists, output $\Gamma_{e}(A^\diamond)\oplus \varnothing$ if $0\in A$ and $\varnothing \oplus \Gamma_{e}(A^\diamond)$ if $0\notin A$.
    \item Otherwise, output $A$.
  \end{enumerate}

  A rough calculation using Theorem 2.6(III) in \cite{kalimullin2003} gives a bound of $\Sigma^0_4(A')$ on the complexity of $F(A)$. The uniqueness (up to
  $e$-degree) of maximal $\ca{K}$-partners guarantees that the function is $e$-invariant, and step 2 ensures that $F$ is not locally constant. Hence, by \hyperref[lemma:kpair-non-uniformity]{Lemma~\ref{lemma:kpair-non-uniformity}}, that $F$ is not uniformly $e$-invariant.

  We need to work a bit harder to show that is not equivalent, on any upper cone, to a Borel uniformly $e$-invariant function. The idea is that, if it were, it would give us a Borel choice function for the enumeration degrees of halves of maximal $\ca{K}$-pairs. With such a choice function, we can define a wellorder of type $\omega$ on each such degree, which leads us to a contradiction. We proceed with the details:

  Suppose towards a contradiction that there are a set $C\subseteq\bb{N}$ and a Borel uniformly $e$-invariant function $H\colon\ca{P}(\bb{N})\to \ca{P}(\bb{N})$ such that $H(X)\equiv_e F(X)$ for all $X\geq_e C$. Fix a Borel code $T$ for $H$ and let \[\operatorname{\ca{K}Max_C}=\{A\mid \exists B\ \exists U\; (C\leq_eU\oplus\overline{U} \mbox{ and }\{A,B\}\mbox{ is a maximal }\ca{K}_{U\oplus\overline{U}}\mbox{-pair}\}.\]

  For $A\in\operatorname{\ca{K}Max_C}$, we know that $H\restriction_{\deg_e(A)}$ is constant by \hyperref[lemma:kpair-non-uniformity]{Lemma~\ref{lemma:kpair-non-uniformity}}. In particular, for all $X\geq_e C$, we have that $A\equiv_e X$ if and only if $H(H(A))=H(H(X))$. Define $R(X)$ to be the least $n\in\bb{N}$ such that $X=\Gamma _n(H(H(X))$. Observe that $R$ is defined for every set $X\geq_eC$ because we always have that $X\equiv_eH(H(X))$. Moreover, $R$ induces a wellorder $\leq$ on $\deg_e(A)$ of type $\omega$ by $X\leq Y$ if $R(X)\leq R(Y)$.

  Denote by $L_{X}$ the left cut induced by the characteristic function of $X$ on the lexicographical ordering of $2^{<\omega}$. Observe that there is a single enumeration operator $\Gamma$ such that $\Gamma (X)=L_X$ for all $X\subseteq\bb{N}$. For a sufficiently $C\oplus T$-generic $G\in 2^{\omega }$, $L_G, \overline{L_G}\nleq_eC\oplus \overline{C}$, and there is some $n\in\bb{N}$ such that $R(L_G\oplus C\oplus \overline{C})=n$. So, by genericity, there must be some $k\in\bb{N}$ such that $G\restriction k \Vdash R(L_G\oplus C\oplus \overline{C})=n$. Let $G^{*}\in 2^{\omega}$ be the infinite binary sequence that results from changing the $k+1$ bit of $G$. That is, $G(j)=G^{*}(j)$ if and only if $j\neq k$. Since we only changed one bit, $G^{*}$ is also generic and $G\equiv_eG^{*}$. Moreover, $L_{G}\equiv_eL_{G^{*}}$. To see this, assume without loss of generality that $G(k)=1$ and $G^{*}(k)=0$, and fix $\sigma=G\restriction k=G^{*}\restriction k$. Then $\tau \in L_{G^{*}}$ if and only if $|\tau |<k+1$ and $\tau<_{lex}\sigma$, or $|\tau|\geq k$ and the string $\tau^{*}\in L_G$, where \[\tau ^{*}(n)=\begin{cases}\tau(n) & \mbox{if }n<|\tau| \mbox{ and } n\neq k-1.\\
    1 & \mbox{if }n=k-1.\end{cases}\]

Putting everything together, we have that $L_{G}\oplus C\oplus \overline{C}, L_{G^{*}}\oplus C\oplus \overline{C} \in \operatorname{\ca{K}Max_C}$ because $L_{G}, L_{G^{*}}$ are semicomputable sets, $H(H(L_G\oplus C\oplus \overline{C}))=H(H(L_{G^{*}}\oplus C\oplus \overline{C}))$ because $L_{G}\equiv_eL_{G^{*}}$, and $R(L_G\oplus C\oplus \overline{C})=R(L_{G^{*}}\oplus C\oplus \overline{C})$ because $G$ and $G^{*}$ are generics with $\sigma\prec G,G^{*}$. However, this is a contradiction as $L_{G}$ and $L_{G^{*}}$ are different sets (one contains $\sigma ^{\smallfrown} 1$ while the other does not), so one of them is not equal to $\Gamma_n(H(H(L_G\oplus C\oplus \overline{C})))$.
  
\end{proof}

\section{Further Questions}

The present work has the intention of starting the investigation of $e$-invariant functions. In this section, we outline some possible avenues for
further research.

We have shown that $\leq_{e}^{\cone}$ is far from a prewellorder even for injective uniformly $e$-invariant functions. The main question that we leave open
is to characterize $\leq_{e}^{\cone}$ up to isomorphism for the class of $e$-invariant functions, or any significant subclass like the uniformly
$e$-invariant functions. This appears to be complicated for two reasons, one global and one local.

The global reason is that understanding the possible cofinal partitions of $\ca{D}_e$ appears necessary to understand $\leq_{e}^{\cone}$. One approach to tackle
this problem is to study the Wadge order on the Scott domain $\ca{P}\omega$. The Scott domain is obtained by putting the topology on $\ca{P}(\bb{N})$ that has as basis
the sets of the form $[D]=\{X\subseteq \bb{N}\mid D\subseteq X\}$ for some $D\subseteq \bb{N}$ finite. This topology is known as the Scott topology and is important in this context because
a function $f\colon\ca{P}(\bb{N})\to\ca{P}(\bb{N})$ is Scott-continuous if and only if it is an enumeration operator relative to some enumeration oracle; that is, there are and
$e$-operator $\Gamma$ and a set $A$ such that for all $X\subseteq\bb{N}$, $f(X)=\Gamma(A\oplus X)$. So, $\ca{P}\omega$ is to enumeration reducibility what
$2^{\omega}$ is to Turing reducibility. Becker \cite{beckerCharacterizationJumpOperators1988} showed, using tools from descriptive set theory, that the preorder
$\leq^{\cone}_T$ on uniformly $T$-invariant functions is isomorphic to the Wadge order on $2^{\omega}$ if you identify every class with its dual. Moreover, Kihara
and Montalbán \cite{kiharaUniformMartinsConjecture2018} improved that to an actual isomorphism by considering uniformly $(m,T)$-invariant functions, where
$m$ stands for many-one reducibility. The recent success by de Brecht \cite{debrechtQuasiPolishSpaces2013} extending the classical results of descriptive set theory from Polish spaces
to some $T_0$ spaces---the quasipolish spaces---has identified $\ca{P}\omega$ as a universal quasipolish space (every quasipolish space is a
{$\mathbf{\Pi}^0_2$} subset of $\ca{P}\omega $) and attracted interest to the Wadge order on $\ca{P}\omega$. While very little is known so far, the Wadge order on the
Scott domain is not very well-behaved (as one might expect).  For example, it is not well-founded and has infinite antichains \cite{duparcWADGEORDERSCOTT2020}.

\begin{question}
  Is the Wadge order on $\ca{P}\omega$ isomorphic to the preorder $\leq_e^{\cone}$ on uniformly $e$-invariant functions?
\end{question}

Shifting to a local approach to understanding $\leq_e^{\cone}$, we find a different complication: the interaction between the enumeration jump and the skip
of a single degree $\tdeg{a}$ is not well-understood. Some immediate restrictions are that $\tdeg{a}^{\diamond }\leq\tdeg{a}'$,
$\tdeg{a}\leq\tdeg{a}^{\diamond \diamond }$, and $\tdeg{a}^{\prime \prime}=\tdeg{a}^{\prime\diamond }$; but there is plenty of variability when we combine the operations. For example, if
$\tdeg{a}$ is cototal, then $\tdeg{a}^{\diamond\prime}=\tdeg{a}^{\prime \prime}$; however, if $\tdeg{a}$ is a skip 2-cycle, then for all $n\in\bb{N}$,
$\left(\tdeg{a}^{\langle n\rangle}\right)'=\tdeg{a}'$, where $\tdeg{a}^{\langle n\rangle}$ denotes the $n$-th iterate of the skip of $\tdeg{a}$. Moreover, there are degrees such that
$\tdeg{a}'<\tdeg{a}^{\diamond \prime}<\tdeg{a}^{\prime \prime}$. A detailed discussion on the behaviour of the skip operator can be found in \cite{Andrews2019OnDegrees}.

\begin{question}
  Let $\tdeg{a}\in\ca{D}_e$. What are the possible order types of
  \[JS(\tdeg{a})=\left\{\tdeg{a}^{\sigma}\mid \sigma \in\{\diamond ,\prime\}^{<\omega}\right\}?\]
\end{question}

Another direction of inquiry comes from the separation between uniformly and non-uniformly $e$-invariant functions. We saw that the combinatorics of
$\ca{K}$-pairs are incompatible with uniformity, but not with $e$-invariance. In \hyperref[cor:unifregresiveconstant]{Corollary~\ref{cor:unifregresiveconstant}}, we saw that regressive uniformly
$e$-invariant functions are trivial (locally constant), but we do not know if the same is true when we drop uniformity.

\begin{question} Is there an $e$-invariant function $f$ such that $f$ is regressive and $f\restriction_{\deg(A)}$ is not constant for all $A\subseteq \omega$?
\end{question}

A different approach to study determinacy in the enumeration degrees was suggested by Slaman (personal communication). The failure of the Cone Theorem
in the enumeration degrees leaves open whether some of the consequences of determinacy can still be proven.  For example, a classical application of
TD shows that there is a cone $\ca{C}$ of Turing degrees such that every cone $\ca{D}_T(\geq\tdeg{d})$ above an element $\tdeg{d}\in\ca{C}$ has the same first-order theory. The same
argument can be used to show that, under TD, there is a total enumeration degree $\tdeg{t}$ such that for all total degrees $\tdeg{d}\geq\tdeg{t}$, the cone of enumeration
degrees $\ca{D}_e(\geq\tdeg{d})$ has the same first-order theory as $\ca{D}_e(\geq\tdeg{t})$. However, this is false for arbitrary enumeration degrees. For example, the proof of
downwards density of $\ca{D}_e$ relativizes to any total degree; thus, for a total degree $\tdeg{t}$, $\ca{D}_e(\geq\tdeg{t})$ has no minimal degrees \cite{gutteridgeResultsEnumerationReductibility1971}. In contrast, by a
theorem of Calhoun and Slaman \cite{calhounP20EnumerationDegrees1996}, above every total degree there is a degree $\tdeg{a}$ such that $\ca{D}_e(\geq\tdeg{a})$ has a minimal degree.

\begin{question}
  Assume TD. Is there a cone $\ca{C}$ of enumeration degrees such that for every cototal degree $\tdeg{a}\in\ca{C}$, $\ca{D}_e(\geq\tdeg{a})$ has the same theory?
\end{question}

The same question might be asked about other classes of degrees. For example, the continuous degrees are strictly between the total and the cototal
degrees under the subset relation and are first-order definable \cite{andrewsCharacterizingContinuousDegrees2019}. Given that all the known proofs that there is a non-total continuous degree are
topological in nature, understanding the theory of the cone above a non-total continuous degree could provide a useful technical tool.

\section*{Acknowledgments}
Thanks to Mariya Soskova for many helpful discussions and advice, to Josiah Jacobsen-Grocott for sharing his insight, to Patrick Lutz for pointing out an omission on the original version of this paper and suggesting an improvement to \hyperref[invnotunif]{Theorem~\ref{invnotunif}}, and to Beth Pcheco for her
thoughtful suggestions. The author was supported by the Secretaría de Ciencia, Humanidades, Tecnología e Innovación of México through scholarship
number 836694.

\printbibliography

\end{document}